\documentclass{amsproc}

\usepackage{amsmath,amssymb,mathtools}
\usepackage{mathrsfs}
\usepackage{microtype}
\usepackage[hidelinks]{hyperref}
\hypersetup{
  pdftitle={An Explicit Logistic Damping Criterion for Boundedness in a Fully Parabolic Keller--Segel System},
  pdfauthor={Jie Jiang},
  pdfsubject={Fully parabolic Keller--Segel systems with quadratic logistic damping},
  pdfkeywords={Keller--Segel system, logistic damping, explicit threshold, Holder regularity, maximal regularity}
}
\usepackage[nameinlink,capitalize,noabbrev]{cleveref}

\allowdisplaybreaks
\numberwithin{equation}{section}

\theoremstyle{plain}
\newtheorem{theorem}{Theorem}[section]
\newtheorem{proposition}[theorem]{Proposition}
\newtheorem{lemma}[theorem]{Lemma}
\newtheorem{corollary}[theorem]{Corollary}

\theoremstyle{remark}
\newtheorem{remark}[theorem]{Remark}

\crefname{theorem}{Theorem}{Theorems}
\crefname{proposition}{Proposition}{Propositions}
\crefname{lemma}{Lemma}{Lemmas}
\crefname{corollary}{Corollary}{Corollaries}
\crefname{remark}{Remark}{Remarks}

\newcommand{\Om}{\Omega}

\newcommand{\pa}{\partial}
\newcommand{\D}{\mathrm D}

\newcommand{\norm}[1]{\left\lVert #1\right\rVert}

\newcommand{\pos}[1]{\left(#1\right)_{+}}

\title[An explicit logistic damping criterion]{An Explicit Logistic Damping Criterion for Boundedness in a Fully Parabolic Keller--Segel System}

\author[J. Jiang]{Jie Jiang}
\address{Wuhan Institute of Physics and Mathematics, Innovation Academy for Precision Measurement Science and Technology, Chinese Academy of Sciences, Wuhan 430071, Hubei Province, P.R. China}
\email{jiang@apm.ac.cn}

\subjclass[2020]{Primary 35K51, 35K55; Secondary 35B35, 92C17}
\keywords{Keller--Segel system, logistic damping, explicit criterion, comparison principle, uniform-in-time boundedness, H\"older regularity}
\date{\today}

\begin{document}

\begin{abstract}
	We study the fully parabolic Keller--Segel system
	\[
	u_t=\Delta u-\chi\nabla\!\cdot(u\nabla v)+\lambda u-\mu u^2,
	\qquad
	\tau v_t=\Delta v-v+u
	\]
	in a bounded smooth convex domain. For every fixed $\tau>0$, we prove that
	\[
	\mu>\frac{N\chi}{4}
	\]
	guarantees global existence and uniform-in-time boundedness. This
	coefficient-explicit sufficient condition is independent of $\tau$ and
	involves no embedding or maximal-regularity constants. To the best of our
	knowledge, it is the first coefficient-explicit boundedness criterion that
	remains unchanged for all $\tau>0$ in arbitrary space dimension.
	
	The proof is built on a new auxiliary comparison function
	\[
	Y_\tau
	=u+\frac{\chi\tau}{2}|\nabla v|^2-(\tau-1)\Delta v,
	\]
	which satisfies a closed scalar parabolic inequality for every $\tau>0$.
	When $\tau\ge1$, this inequality yields a direct pointwise comparison and
	an explicit bound for $u$. When $0<\tau<1$, it instead provides a
	uniform upper bound for $v$. Applying a parabolic squeezing argument to
	the transform $z=e^{-\chi v/2}$ then yields a uniform H\"older bound for
	$v$. H\"older--Sobolev interpolation and weighted maximal
	$L^p$-regularity subsequently give an $L^p$-bound for $u$ with
	sufficiently large $p$, and standard parabolic smoothing closes the
	argument.
\end{abstract}

\maketitle

\section{Introduction and main result}

Let $\Omega\subset\mathbb R^N$, $N\ge1$, be a bounded smooth domain. We
consider the fully parabolic Keller--Segel system with quadratic logistic
damping
\begin{equation}
	\label{eq:system}
	\left\{
	\begin{aligned}
		u_t&=\Delta u-\chi\nabla\!\cdot(u\nabla v)+\lambda u-\mu u^2,
		&&x\in\Omega,\ t>0,\\
		\tau v_t&=\Delta v-v+u,
		&&x\in\Omega,\ t>0,\\
		\partial_\nu u&=\partial_\nu v=0,
		&&x\in\partial\Omega,\ t>0,\\
		u(\cdot,0)&=u_0,\qquad v(\cdot,0)=v_0,
		&&x\in\Omega,
	\end{aligned}
	\right.
\end{equation}
where $\chi,\mu,\tau>0$ and $\lambda\in\mathbb R$. Here $u$ denotes the
cell density and $v$ the concentration of the chemoattractant. The attractive
cross-diffusion tends to concentrate cells, whereas the term $-\mu u^2$
penalizes high densities. Determining how strong this quadratic degradation
must be in order to prevent chemotactic concentration is therefore a natural
quantitative problem. In the normalization \eqref{eq:system}, the cell
diffusivity equals one and the effective signal diffusivity equals $1/\tau$;
hence $\tau=1$ is the equal-diffusivity case. We refer to
\cite{BellomoBellouquidTaoWinkler2015} for general background on
Keller--Segel models and chemotaxis-growth systems.

In the parabolic--elliptic setting, Tello and Winkler
\cite{TelloWinkler2007} showed that sufficiently strong logistic damping
prevents blow-up. For the fully parabolic problem, Winkler
\cite{Winkler2010} proved global boundedness on smooth bounded convex
domains for every fixed $\tau>0$, provided that $\mu$ is sufficiently
large. In the equal-diffusion case $\tau=1$, he further observed that
\[
Z:=\frac12|\nabla v|^2+\frac1\chi u
\]
satisfies a closed scalar parabolic inequality, from which the comparison
principle yields the explicit condition $\mu>N\chi/4$. For $\tau\ne1$,
however, the mismatch of diffusion rates prevents this particular scalar closure, and
the largeness condition obtained by the mixed energy argument is only implicit.

Most subsequent refinements concerned the equal-diffusion case
$\tau=1$. On general smooth three-dimensional domains, Tao and
Winkler \cite{TaoWinkler2015ZAMP} obtained the explicit condition
$\mu\ge23\chi$ for a more general chemotaxis--Stokes system, which
contains the fluid-free Keller--Segel system as a special case. Lin
and Mu \cite{LinMu2016} subsequently improved this condition to
$\mu>20\chi$ for the Keller--Segel system. Zheng \cite{Zheng2017}
extended the pointwise comparison framework of \cite{Winkler2010} to
damping terms of the form $-\mu u^r$, $r\ge2$, and obtained explicit
ultimate bounds. Xiang \cite{Xiang2018SIAM} further improved the
explicit coefficient condition in three dimensions and quantified
the dependence of the resulting uniform bounds on $\chi$ and $\mu$.
Zheng et al. \cite{ZhengEtAl2018} established boundedness on general
smooth domains in arbitrary dimensions, at the cost of a condition
involving a maximal Sobolev regularity constant.

A coefficient-explicit treatment allowing unequal diffusion rates was
given by Xiang \cite{Xiang2018JMAA} for the more general system with
diffusion coefficients $d_1,d_2$ and signal production rate $\alpha$. He
recovered
\(
\mu>\frac{N\alpha\chi}{4d_1}
\)
when $d_1=d_2$ and the domain is convex; more involved
diffusion-dependent conditions for unequal diffusivities in dimensions
$3\le N\le5$ were also derived there.

To the best of our knowledge, it remained open whether the
equal-diffusion condition
\(
\mu>\frac{N\chi}{4}
\)
continues to ensure global boundedness, without any dependence on
$\tau$, in arbitrary space dimension for every fixed $\tau>0$. We answer
this question affirmatively on bounded smooth convex domains.

\begin{theorem}
	\label{thm:main}
	Let $\Omega\subset\mathbb R^N$ be a bounded convex domain with smooth
	boundary. Assume that $\chi>0$, $\tau>0$, $\lambda\in\mathbb R$, and
	\begin{equation}
		\label{eq:threshold}
		\mu>\frac{N\chi}{4}.
	\end{equation}
	Let $u_0\in C(\overline\Omega)$ and
	$v_0\in W^{1,\infty}(\Omega)$ be nonnegative. Then
	\eqref{eq:system} possesses a unique nonnegative global classical solution
	satisfying
	\begin{equation}
		\label{eq:mainbounded}
		\sup_{t>0}\left(
		\|u(t)\|_{L^\infty(\Omega)}
		+\|v(t)\|_{W^{1,\infty}(\Omega)}
		\right)<\infty.
	\end{equation}
	Set
	\begin{equation}
		\label{eq:deltaCstar}
		\delta:=\mu-\frac{N\chi}{4},
		\qquad
		C_*:=\frac{(\lambda+\tau^{-1})_+^2}{4\delta}.
	\end{equation}
	If $\tau>1$, then
\begin{equation}
	\label{eq:explicitabsorbing}
	\limsup_{t\to\infty}\|u(t)\|_{L^\infty(\Omega)}
	\le
	\frac{\bigl((\tau-1)\lambda+\chi\tau C_*\bigr)_+}
	{(\tau-1)\mu+\chi},
\end{equation}
	whereas, if $\tau=1$, then
	\begin{equation}
		\label{eq:tau1absorbing}
		\limsup_{t\to\infty}\|u(t)\|_{L^\infty(\Omega)}
		\le C_*.
	\end{equation}
\end{theorem}

\begin{remark}[General constant coefficients]
	\label{rem:general-coefficients}
	Consider
	\[
	u_t=d_1\Delta u-\chi\nabla\!\cdot(u\nabla v)
	+\lambda u-\mu u^2,
	\qquad
	v_t=d_2\Delta v-\beta v+\alpha u,
	\]
	where \(d_1,d_2,\alpha,\beta,\chi,\mu>0\) and
	\(\lambda\in\mathbb R\). Under the change of variables
	\[
	y=\sqrt{\frac{\beta}{d_2}}\,x,
	\qquad
	s=\frac{\beta d_1}{d_2}t,
	\qquad
	U=u,
	\qquad
	V=\frac{\beta}{\alpha}v,
	\]
	the system reduces, on the rescaled domain, to \eqref{eq:system} with
	\[
	\widehat\chi=\frac{\alpha\chi}{\beta d_1},
	\qquad
	\widehat\lambda=\frac{d_2\lambda}{\beta d_1},
	\qquad
	\widehat\mu=\frac{d_2\mu}{\beta d_1},
	\qquad
	\widehat\tau=\frac{d_1}{d_2}.
	\]
	Therefore, \cref{thm:main} yields the coefficient-explicit condition
	\[
	\mu>\frac{N\alpha\chi}{4d_2}.
	\]
	When $d_1=d_2$, this reduces to the equal-diffusion criterion obtained
	by Xiang \cite{Xiang2018JMAA}; no equality between the diffusion
	coefficients is required here.
\end{remark}

The proof is built around the new auxiliary comparison function
\[
Y_\tau:=u+\frac{\chi\tau}{2}|\nabla v|^2-(\tau-1)\Delta v,
\]
which coincides with $\chi Z$ when $\tau=1$. The correction term
$-(\tau-1)\Delta v$ compensates exactly for the mismatch between the two
diffusion rates. A direct calculation, together with
$|D^2v|^2\ge N^{-1}(\Delta v)^2$, gives
\[
\begin{aligned}
	\left(\partial_t-\frac1\tau\Delta+\frac1\tau\right)Y_\tau
	&\le
	-\frac\chi N\left(\Delta v+\frac N2u\right)^2
	-\frac\chi2|\nabla v|^2
	-\delta u^2+\left(\lambda+\frac1\tau\right)u\le C_*,
\end{aligned}
\]
where $\delta$ and $C_*$ are the constants given in \cref{thm:main}.
Moreover, the convexity of $\Omega$ ensures that
$\partial_\nu Y_\tau\le0$. The comparison principle therefore yields a
uniform-in-time upper bound for $Y_\tau$ for each $\tau>0$. The above
completion of the square is precisely where the threshold
$\mu>N\chi/4$ enters the proof.

The subsequent argument splits according to the relaxation regime.
When $\tau=1$, boundedness follows immediately from $u\le Y_1$. When
$\tau>1$, the upper bound for $Y_\tau$ gives a favorable upper estimate
for $-\Delta v$, which reduces the $u$-equation to a scalar logistic
inequality. When $0<\tau<1$, the same comparison estimate first yields an
upper bound for $v$. A parabolic squeezing argument applied to
$e^{-\chi v/2}$ then gives uniform H\"older regularity of $v$.
H\"older--Sobolev interpolation, maximal $L^p$-regularity, and standard
parabolic smoothing finally yield boundedness of $u$.

The remainder of the paper is organized as follows. Section~2 collects the
local solvability and boundary properties needed in the proof. Section~3
derives the differential inequality for $Y_\tau$, which is valid for every
$\tau>0$. The regime $\tau\ge1$ is treated in Section~4, while the
parabolic squeezing and regularity arguments for $0<\tau<1$ are developed
in Section~5.

\section{Preliminaries}
In this section, we collect some useful lemmas. We first recall the local well-posedness and extensibility criterion for
classical solutions. This standard result follows from parabolic theory;
see, for instance, \cite[Lemma~1.1]{Winkler2010}.
\begin{lemma}
	\label{lem:local}
	Let $\Omega\subset\mathbb R^N$ be a bounded smooth domain, and let
	$\chi,\mu,\tau>0$ and $\lambda\in\mathbb R$. Suppose that
	\[
	u_0\in C(\overline\Omega),
	\qquad
	v_0\in W^{1,\infty}(\Omega)
	\]
	are nonnegative. Then there exist a maximal existence time
	$T_{\max}\in(0,\infty]$ and a unique nonnegative classical solution
	$(u,v)$ of \eqref{eq:system} such that
	\[
	u,v\in
	C^0\bigl(\overline\Omega\times[0,T_{\max})\bigr)
	\cap
	C^{2,1}\bigl(\overline\Omega\times(0,T_{\max})\bigr).
	\]
	Moreover, if $T_{\max}<\infty$, then
	\begin{equation}
		\label{eq:extensibility-u}
		\limsup_{t\nearrow T_{\max}}
		\|u(\cdot,t)\|_{L^\infty(\Omega)}
		=\infty.
	\end{equation}
	If $u_0\not\equiv0$, then $u>0$ in
	$\overline\Omega\times(0,T_{\max})$.
\end{lemma}

We also need the following boundary properties.
\begin{lemma}
	\label{lem:boundary}
	Let $\Omega\subset\mathbb R^N$ be a bounded smooth convex domain.
	
	\begin{enumerate}
		\item
		If $\varphi\in C^3(\overline\Omega)$ satisfies
		$\partial_\nu\varphi=0$ on $\partial\Omega$, then
		\begin{equation}
			\label{eq:boundary-gradient}
			\partial_\nu |\nabla\varphi|^2\le0
			\qquad\text{on }\partial\Omega.
		\end{equation}
		
		\item
		Let $(u,v)$ be the maximal classical solution furnished by
		\cref{lem:local}. Then, for every $0<t_0<T<T_{\max}$,
		\begin{equation}
		\label{eq:boundary-laplacian}
		\partial_\nu\Delta v=0
		\qquad\text{on }\partial\Omega\times[t_0,T].
		\end{equation}
	\end{enumerate}
\end{lemma}

\begin{proof}

	The inequality \eqref{eq:boundary-gradient} is a standard consequence
	of the convexity of $\Omega$ and the homogeneous Neumann condition.
	
By standard positive-time parabolic smoothing, the solution furnished by
\cref{lem:local} is smooth on
\(\overline\Omega\times[t_0,T]\) for every
\(0<t_0<T<T_{\max}\). 	For $t>t_0$, parabolic regularity permits differentiation of
	$\partial_\nu v=0$ with respect to time, and hence
	$\partial_\nu v_t=0$. Since
	\[
	\Delta v=\tau v_t+v-u,
	\]
	the Neumann boundary conditions for $u$, $v$ and $v_t$ imply
	\eqref{eq:boundary-laplacian}. 
\end{proof}

\section{The auxiliary comparison function and its differential inequality}

We now derive the comparison estimate common to all $\tau>0$.
Let $(u,v)$ be a smooth solution of \eqref{eq:system}. In view of the positive-time smoothing mentioned in the proof
of \cref{lem:boundary}, all calculations below are justified for \(t>0\). 

Define the key auxiliary comparison function
\begin{equation}
	\label{eq:Ytau-intro}
	Y_\tau:=u+\frac{\chi\tau}{2}|\nabla v|^2-(\tau-1)\Delta v.
\end{equation}
By  \cref{lem:boundary}
\begin{equation}
	\label{eq:boundary-Y}
	\partial_\nu Y_\tau\le0
	\qquad\text{on }
	\partial\Omega\times(t_0,T).
\end{equation}
Moreover, $Y_\tau$ satisfies an evolution identity.
\begin{lemma}
\label{lem:Yidentity}
For all   \((x,t)\in\Omega\times(t_0,T)\), it holds 
\begin{equation}
\label{eq:Yidentity}
 \pa_tY_\tau-\frac1\tau\Delta Y_\tau+\frac1\tau Y_\tau
 =-\chi|\D^2v|^2-\frac\chi2|\nabla v|^2
   -\chi u\Delta v-\mu u^2
   +\left(\lambda+\frac1\tau\right)u.
\end{equation}
\end{lemma}

\begin{proof}
Set $q=|\nabla v|^2/2$. Taking the scalar product of the gradient of the
second equation in \eqref{eq:system} with $\nabla v$ and using the Bochner
identity gives
\begin{equation}
\label{eq:q-evolution}
 \tau q_t
 =\Delta q-|\D^2v|^2-|\nabla v|^2+\nabla u\cdot\nabla v.
\end{equation}
Moreover,
\begin{equation}
\label{eq:Deltav-evolution}
 \tau(\Delta v)_t=\Delta^2v-\Delta v+\Delta u,
\end{equation}
and the first equation can be written as
\begin{equation}
\label{eq:u-expanded}
 u_t=\Delta u-\chi\nabla u\cdot\nabla v-\chi u\Delta v
     +\lambda u-\mu u^2.
\end{equation}
As a result,
\begin{align*}
 \pa_tY_\tau
 ={}&\Delta u-\chi\nabla u\cdot\nabla v-\chi u\Delta v
      +\lambda u-\mu u^2\\
 &+\chi\Delta q-\chi|\D^2v|^2-\chi|\nabla v|^2
      +\chi\nabla u\cdot\nabla v\\
 &-\frac{\tau-1}{\tau}
      \bigl(\Delta^2v-\Delta v+\Delta u\bigr),
\end{align*}
whereas
\[
 \frac1\tau\Delta Y_\tau
 =\frac1\tau\Delta u+\chi\Delta q
   -\frac{\tau-1}{\tau}\Delta^2v.
\]
 Thus
\[
 \pa_tY_\tau-\frac1\tau\Delta Y_\tau
 =-\chi|\D^2v|^2-\chi|\nabla v|^2-\chi u\Delta v
  -\mu u^2+\lambda u+\frac{\tau-1}{\tau}\Delta v,
\]
which gives rise to \eqref{eq:Yidentity}  by adding $\tau^{-1}Y_\tau$ to both sides.
\end{proof}
We arrive at a closed scalar differential inequality.
\begin{lemma}
\label{lem:Yscalar}
Under \eqref{eq:threshold}, $Y_\tau$ satisfies
\begin{equation}
\label{eq:Yscalar}
 \pa_tY_\tau-\frac1\tau\Delta Y_\tau+\frac1\tau Y_\tau
 \le C_*,
\end{equation}
where $C_*$ is defined in \eqref{eq:deltaCstar}.
\end{lemma}

\begin{proof}
Applying the matrix inequality $|\D^2v|^2\ge N^{-1}(\Delta v)^2$, one finds that
\begin{align}
 -\chi|\D^2v|^2-\chi u\Delta v-\mu u^2
 &\le-\frac\chi N(\Delta v)^2-\chi u\Delta v-\mu u^2\notag\\
 &=-\frac\chi N\left(\Delta v+\frac N2u\right)^2
   -\left(\mu-\frac{N\chi}{4}\right)u^2\notag\\
 &\le-\delta u^2.
\label{eq:squarecompletion}
\end{align}
Consequently, \cref{lem:Yidentity} and nonnegativity of $u$ yield
\[
 \pa_tY_\tau-\frac1\tau\Delta Y_\tau+\frac1\tau Y_\tau
 \le-\delta u^2+\left(\lambda+\frac1\tau\right)u
 \le\frac{\pos{\lambda+\tau^{-1}}^2}{4\delta}=C_*.
\]
This completes the proof.
\end{proof}
Combining this differential inequality with the boundary condition
\eqref{eq:boundary-Y} yields the desired comparison estimate.
\begin{proposition}
\label{prop:Ybound}
Let $0<t_0<T<T_{\max}$. Then
\begin{equation}
\label{eq:Ybound}
 Y_\tau(x,t)
 \le e^{-(t-t_0)/\tau}
       \max_{\overline\Om}\pos{Y_\tau(\cdot,t_0)}
     +\tau C_*\bigl(1-e^{-(t-t_0)/\tau}\bigr)
\end{equation}
for all $(x,t)\in\overline\Om\times[t_0,T]$. In particular, if the solution
is global, then
\begin{equation}
\label{eq:Ylimsup}
 \limsup_{t\to\infty}\sup_{x\in\Om}Y_\tau(x,t)\le\tau C_*.
\end{equation}
\end{proposition}

\begin{proof}
	By \eqref{eq:boundary-Y} and \cref{lem:Yscalar}, we have
	\[
	\partial_t Y_\tau-\frac1\tau\Delta Y_\tau
	+\frac1\tau Y_\tau\le C_*
	\quad\text{in }\Omega\times(t_0,T),
	\]
	and
	\[
	\partial_\nu Y_\tau\le0
	\quad\text{on }\partial\Omega\times(t_0,T).
	\]
 The parabolic comparison principle therefore yields
	\eqref{eq:Ybound}. If the solution is global, then letting
	\(t\to\infty\) gives \eqref{eq:Ylimsup}.
\end{proof}
For later use, fix $t_0>0$ and write
\begin{equation}
\label{eq:M-t0}
 M_{t_0}:=
 \max\left\{
 \max_{\overline\Om}\pos{Y_\tau(\cdot,t_0)},\ \tau C_*
 \right\}.
\end{equation}
Then
\begin{equation}
\label{eq:Y-M}
 Y_\tau(x,t)\le M_{t_0}
 \qquad(x,t)\in\overline\Om\times[t_0,T_{\max}).
\end{equation}

\section{The regime \texorpdfstring{$\tau\ge1$}{tau >= 1}}

For $\tau\ge1$, the upper bound for $Y_\tau$ leads to pointwise control
of the density: directly when $\tau=1$, and through a scalar logistic
comparison when $\tau>1$.

\begin{lemma}
\label{lem:logistic-comparison}
Let $\tau>1$ and suppose that $Y_\tau\le M$ on
$\Om\times[t_0,T)$. Then
\begin{equation}
\label{eq:minusDeltav}
 -\Delta v
 \le\frac{M}{\tau-1}-\frac{u}{\tau-1}
     -\frac{\chi\tau}{2(\tau-1)}|\nabla v|^2
\end{equation}
and
\begin{equation}
\label{eq:u-logistic-PDE}
 u_t
 \le\Delta u-\chi\nabla v\cdot\nabla u
 +\left(\lambda+\frac{\chi M}{\tau-1}\right)u
 -\left(\mu+\frac\chi{\tau-1}\right)u^2.
\end{equation}
Consequently,
\begin{equation}
	\label{eq:u-explicit-M}
	\sup_{t\in[t_0,T)}\norm{u(t)}_{L^\infty}
	\le
	\max\left\{
	\norm{u(t_0)}_{L^\infty},
	\frac{\pos{(\tau-1)\lambda+\chi M}}
	{(\tau-1)\mu+\chi}
	\right\}.
\end{equation}
\end{lemma}

\begin{proof}
The estimate \eqref{eq:minusDeltav} follows directly from
$Y_\tau=u+(\chi\tau/2)|\nabla v|^2-(\tau-1)\Delta v\le M$. Inserting it
into
\[
 u_t=\Delta u-\chi\nabla v\cdot\nabla u
      +\chi u(-\Delta v)+\lambda u-\mu u^2
\]
and discarding the additional nonpositive term
$-\chi^2\tau u|\nabla v|^2/[2(\tau-1)]$ gives
\eqref{eq:u-logistic-PDE}.

Let $U$ solve
\[
 U'=aU-bU^2,
 \qquad U(t_0)=\norm{u(t_0)}_{L^\infty},
\]
with
\[
 a=\lambda+\frac{\chi M}{\tau-1},
 \qquad
 b=\mu+\frac\chi{\tau-1}>0.
\]
The spatially constant function $U$ is a supersolution of
\eqref{eq:u-logistic-PDE}. The parabolic comparison principle and the
standard logistic bound prove \eqref{eq:u-explicit-M}.
\end{proof}
We can now complete the argument for $\tau\ge1$, including the explicit
ultimate bounds.
\begin{proposition}
	\label{prop:tau-geq1}
	Under the assumptions of \cref{thm:main}, suppose that \(\tau\ge1\).
	Then the solution is global and satisfies
	\[
	\sup_{t>0}
	\left(
	\|u(t)\|_{L^\infty(\Omega)}
	+\|v(t)\|_{W^{1,\infty}(\Omega)}
	\right)<\infty.
	\]
	Moreover, if \(\tau>1\), then
	\begin{equation}
		\label{eq:prop-explicitabsorbing}
		\limsup_{t\to\infty}\|u(t)\|_{L^\infty(\Omega)}
		\le
		\frac{\pos{(\tau-1)\lambda+\chi \tau C_*}}
		{(\tau-1)\mu+\chi},
	\end{equation}
	whereas, if \(\tau=1\), then
	\begin{equation}
		\label{eq:prop-tau1absorbing}
		\limsup_{t\to\infty}\|u(t)\|_{L^\infty(\Omega)}
		\le C_*.
	\end{equation}
\end{proposition}

\begin{proof}
	Assume first that \(\tau>1\). Fix
	\(t_0\in(0,T_{\max})\), and let \(M_{t_0}\) be defined by
	\eqref{eq:M-t0}. By \eqref{eq:Y-M},
	\[
	Y_\tau(x,t)\le M_{t_0}
	\qquad
	\text{for all }
	(x,t)\in\overline\Omega\times[t_0,T_{\max}).
	\]
	Applying \cref{lem:logistic-comparison} we obtain \eqref{eq:u-explicit-M} with \(M\) being replaced by \(M_{t_0}\), which together with the local boundedness of \(u\) on \([0,t_0]\), yields
	\begin{equation}
		\label{eq:u-bound-premax}
		\sup_{t\in(0,T_{\max})}
		\|u(t)\|_{L^\infty(\Omega)}<\infty.
	\end{equation}
	The extensibility criterion in \cref{lem:local} therefore implies
	\(T_{\max}=\infty\).
	
We next prove \eqref{eq:prop-explicitabsorbing}. Let \(\varepsilon>0\).
By \eqref{eq:Ylimsup}, there exists \(T_\varepsilon>0\) such that
\[
Y_\tau(x,t)\le\tau C_*+\varepsilon
\qquad
\text{for all }(x,t)\in
\overline\Omega\times[T_\varepsilon,\infty).
\]
Applying \cref{lem:logistic-comparison} with
\(M=\tau C_*+\varepsilon\), and comparing with the corresponding
scalar logistic equation starting from
\(\|u(T_\varepsilon)\|_{L^\infty(\Omega)}\), we obtain
\[
\limsup_{t\to\infty}\|u(t)\|_{L^\infty(\Omega)}
\le
\frac{\bigl((\tau-1)\lambda+\chi(\tau C_*+\varepsilon)\bigr)_+}
{(\tau-1)\mu+\chi}.
\]
Letting \(\varepsilon\searrow0\) proves
\eqref{eq:prop-explicitabsorbing}.
	
	Suppose now that \(\tau=1\). In this case,
	\begin{equation}
		\label{eq:Ytau1}
		Y_1
		=
		u+\frac{\chi}{2}|\nabla v|^2
		\ge u.
	\end{equation}
	Fixing \(t_0\in(0,T_{\max})\), we infer from
	\cref{prop:Ybound} that \(Y_1\), and hence \(u\), is uniformly bounded
	on \(\Omega\times[t_0,T_{\max})\). Together with local boundedness on
	\([0,t_0]\), this again gives \eqref{eq:u-bound-premax}, and therefore
	\(T_{\max}=\infty\). Moreover,	\eqref{eq:prop-tau1absorbing} follows
	immediately from \eqref{eq:Ytau1} and \eqref{eq:Ylimsup}.
	
	Finally, in both cases \(\tau\ge1\), the uniform \(L^\infty\)-bound for
	\(u\) together with the Neumann heat semigroup estimates yields
	\[
	\sup_{t>0}\|v(t)\|_{W^{1,\infty}(\Omega)}<\infty.
	\]
	This completes the proof.
\end{proof}

\section{The regime \texorpdfstring{$0<\tau<1$}{0<tau<1}}

Throughout this section, let \(0<\tau<1\), fix
\(t_0\in(0,T_{\max})\), and set
\[
M:=M_{t_0},
\]
where \(M_{t_0}\) is defined in \eqref{eq:M-t0}.

\subsection{Pointwise control and parabolic squeezing}
When $0<\tau<1$, the sign of the Laplacian term in $Y_\tau$ no longer
gives direct control of $u$, but it does yield a pointwise upper bound
for the signal.
\begin{lemma}
	\label{lem:v-upper}
	We have
	\begin{equation}
		\label{eq:v-uniform-upper}
		0\le v(x,t)\le v^*
		:=
		\max\left\{
		\|v(t_0)\|_{L^\infty(\Omega)},
		\frac{M}{1-\tau}
		\right\}
	\end{equation}
	for all
	\((x,t)\in\overline\Omega\times[t_0,T_{\max})\).
\end{lemma}

\begin{proof}
	By \eqref{eq:Y-M},
	\[
	u+\frac{\chi\tau}{2}|\nabla v|^2
	+(1-\tau)\Delta v\le M,
	\]
	which together with \(\tau v_t+v=\Delta v+u\),  yields
	\[
	\tau v_t+v
	\le
	\frac{M}{1-\tau}
	-\frac{\tau}{1-\tau}u
	-\frac{\chi\tau}{2(1-\tau)}|\nabla v|^2
	\le\frac{M}{1-\tau}.
	\]
	Pointwise comparison with the corresponding scalar ordinary
	differential equation gives
	\[
	v(x,t)
	\le
	e^{-(t-t_0)/\tau}\|v(t_0)\|_{L^\infty(\Omega)}
	+\frac{M}{1-\tau}
	\bigl(1-e^{-(t-t_0)/\tau}\bigr)
	\le v^*.
	\]
	The lower bound follows from \(v\ge0\).
\end{proof}

Define
\begin{equation}
	\label{eq:zdef}
	z:=e^{-\chi v/2}.
\end{equation}

\begin{lemma}
	\label{lem:z-squeezing}
	The function \(z\) satisfies
	\[
	0<z\le1,
	\qquad
	\partial_\nu z=0
	\quad\text{on }\partial\Omega\times(t_0,T_{\max}),
	\]
	and 
	\begin{equation}
		\label{eq:z-lowerheat}
		z_t-\Delta z\ge-C_1,
		\qquad
		C_1:=\frac{\chi M}{2\tau},
	\end{equation}
	as well as
	\begin{equation}
		\label{eq:z-upperheat}
		z_t-\frac1\tau\Delta z\le C_2,
		\qquad
		C_2:=\frac1{e\tau},
	\end{equation}	in \(\Omega\times(t_0,T_{\max})\).
\end{lemma}

\begin{proof}
	Since \(v\ge0\), we have \(0<z\le1\), and on \(\partial\Omega\times(t_0,T_{\max})\)
	\[
	\partial_\nu z=-\frac\chi2z\,\partial_\nu v=0.
	\]
	Substituting 
	\[
	u=\tau v_t-\Delta v+v
	\]
	into the definition of \(Y_\tau\), we obtain
	\[
	Y_\tau
	=
	\tau\left(
	v_t-\Delta v+\frac\chi2|\nabla v|^2+\frac v\tau
	\right).
	\]
	Hence \(Y_\tau\le M\) implies
	\[
	v_t-\Delta v+\frac\chi2|\nabla v|^2
	\le\frac{M-v}{\tau}.
	\]
	Since
	\[
	z_t-\Delta z
	=
	-\frac\chi2z
	\left(
	v_t-\Delta v+\frac\chi2|\nabla v|^2
	\right),
	\]
	it follows that
	\[
	z_t-\Delta z
	\ge
	-\frac{\chi M}{2\tau}z+\frac{\chi}{2\tau}vz
	\ge-\frac{\chi M}{2\tau},
	\]
	which proves \eqref{eq:z-lowerheat}.
	
	On the other hand, the signal equation gives
	\[
	\tau z_t
	=
	\Delta z-\frac{\chi^2}{4}z|\nabla v|^2
	+\frac\chi2z(v-u).
	\]
	Therefore, using \(u\ge0\),
	\[
	z_t-\frac1\tau\Delta z
	\le\frac{\chi}{2\tau}vz
	\le\frac{\chi}{2\tau}
	\sup_{s\ge0}se^{-\chi s/2}
	=\frac1{e\tau},
	\]
	which proves \eqref{eq:z-upperheat}.
\end{proof}

The two inequalities in \cref{lem:z-squeezing} involve different
diffusion rates. The following elementary lemma combines them into a
single uniformly parabolic equation with bounded measurable
coefficients.

\begin{lemma}
	\label{lem:squeezing}
	Let
	\[
	Q:=\Omega\times(t_0,T_{\max})
	\]
	and let $z\in C^{2,1}(Q)$ satisfy
	\[
	z_t-\Delta z\ge-C_1,
	\qquad
	z_t-\frac1\tau\Delta z\le C_2
	\quad\text{in }Q,
	\]
	where $0<\tau<1$. Then there exist measurable functions
	$a,f:Q\to\mathbb R$ such that
	\[
	z_t-a(x,t)\Delta z=f(x,t)
	\quad\text{in }Q,
	\]
	and
	\[
	1\le a(x,t)\le\frac1\tau,
	\qquad
	|f(x,t)|\le C_0:=\max\{C_1,C_2\}.
	\]
\end{lemma}

\begin{proof}
	Set
	\[
	P:=z_t,\qquad R:=\Delta z,\qquad
	k:=\frac1\tau>1,
	\]
	and
	\[
	r:=P-R,\qquad s:=P-kR.
	\]
	Then
	\[
	r\ge-C_0,\qquad s\le C_0.
	\]
	Let
	\[
	E:=\{(x,t)\in Q:r(x,t)>C_0\}
	\]
	and define
	\[
	a:=
	\begin{cases}
		1,&\text{in }Q\setminus E,\\[1mm]
		\displaystyle\frac{P-C_0}{R},&\text{in }E,
	\end{cases}
	\qquad
	f:=P-aR.
	\]
	
	In $Q\setminus E$, we have $a=1$ and
	\[
	f=r\in[-C_0,C_0].
	\]
	In $E$, since $s\le C_0<r$,
	\[
	(k-1)R=r-s>0,
	\]
	and hence $R>0$. Thus the quotient defining $a$ is well-defined on
	$E$. Moreover,
	\[
	a-1=\frac{r-C_0}{R}>0,
	\qquad
	k-a=\frac{C_0-s}{R}\ge0,
	\]
	and
	\[
	f=P-aR=C_0.
	\]
	Consequently,
	\[
	1\le a\le k=\frac1\tau,
	\qquad |f|\le C_0,\qquad \text{in  }Q.
	\]
Finally, $E$ is measurable, and the quotient $(P-C_0)/R$ is
	measurable on $E$, where $R>0$. Hence $a$ and $f$ are both measurable.
\end{proof}

\subsection{H\"older regularity of the signal}
The preceding squeezing argument allows us to apply H\"older estimates for uniformly parabolic equations with bounded measurable coefficients.
\begin{proposition}
	\label{prop:z-v-holder}
	There exists $\alpha=\alpha(N,\tau,\Omega)\in(0,1)$ such that,
	for every $t_1\in(t_0,T_{\max})$, there exists $C>0$ satisfying
	\begin{equation}
		\label{eq:z-v-holder}
		\sup_{T\in(t_1,T_{\max})}
		\left(
		\|z\|_{C^{\alpha,\alpha/2}
			(\overline\Omega\times[t_1,T])}
		+
		\|v\|_{C^{\alpha,\alpha/2}
			(\overline\Omega\times[t_1,T])}
		\right)
		\le C.
	\end{equation}
\end{proposition}

\begin{proof}
	By \cref{lem:squeezing,lem:z-squeezing}, there exist measurable
	functions \(a\) and \(f\) on
	\(\Omega\times(t_0,T_{\max})\) such that
	\[
	z_t-a(x,t)\Delta z=f,
	\qquad
	1\le a(x,t)\le\tau^{-1},
	\qquad
	|f|\le C_0,
	\]
	together with
	\[
	0<z\le1,
	\qquad
	\partial_\nu z=0.
	\]
	
	Fix \(t_1\in(t_0,T_{\max})\), and fix
	\(\delta_0=1/2\). Choose \(\rho>0\), depending only on
	\(t_1-t_0\) and \(\Omega\), sufficiently small that
	\[
	16\rho^2<t_1-t_0,
	\]
	every boundary portion of radius \(4\rho\) is contained in a
	boundary coordinate neighborhood, and the boundary estimate cited
	below is applicable at scale \(4\rho\). Thus every backward cylinder
	of radius \(4\rho\) with top time in
	\([t_1,T_{\max})\) lies above the time level \(t_0\).
	
	We apply the interior and boundary H\"older estimates
	\cite[Corollaries~7.41 and~7.51]{Lieberman2005} to
	\[
	Lz:=-z_t+a^{ij}(x,t)D_{ij}z=-f,
	\qquad
	a^{ij}(x,t):=a(x,t)\delta_{ij}.
	\]
	Apart from the weighted assumption on the inhomogeneous term in the
	boundary estimate, the structural hypotheses follow immediately
	from \(1\le a\le\tau^{-1}\), the smoothness of
	\(\partial\Omega\), and the homogeneous Neumann condition. Indeed,
	in the notation of \cite{Lieberman2005}, the latter has the form
	\[
	\mathcal Mz=\beta\cdot Dz+\beta^0z=0,
	\qquad
	\beta=\gamma,
	\qquad
	\beta^0=0,
	\]
	where \(\gamma\) denotes the unit inner normal.
	
	It remains to verify the condition on \(f\). Let \(d(X)\) denote
	the distance to the lateral boundary occurring in
	\cite[Corollary~7.51]{Lieberman2005}. On each boundary cylinder of
	radius \(4\rho\),
	\[
	0<d(X)\le4\rho.
	\]
	Since \(\delta_0-1=-1/2\), the bound \(|f|\le C_0\) yields
	\[
	|f(X)|
	\le C_0
	\le C_0(4\rho)^{1-\delta_0}d(X)^{\delta_0-1}.
	\]
	Hence the required weighted bound on the inhomogeneous term is
	satisfied.
	
	Corollaries~7.41 and~7.51 of \cite{Lieberman2005} therefore provide
	some \(\alpha=\alpha(N,\tau,\Omega)\in(0,1)\) and a uniform local
	\(C^{\alpha,\alpha/2}\)-estimate for \(z\) on
	\(\overline\Omega\times[t_1,T_{\max})\), with constants independent
	of the time position of the cylinders. A standard covering
	argument, together with \(0<z\le1\), gives
	\[
	\sup_{T\in(t_1,T_{\max})}
	\|z\|_{C^{\alpha,\alpha/2}
		(\overline\Omega\times[t_1,T])}
	\le C.
	\]
	
	Finally, by \eqref{eq:v-uniform-upper},
	\[
	e^{-\chi v^*/2}\le z\le1,
	\qquad
	v=-\frac2\chi\log z.
	\]
	Since \(s\mapsto-(2/\chi)\log s\) is Lipschitz on
	\([e^{-\chi v^*/2},1]\), we obtain
	\[
	[v]_{C^{\alpha,\alpha/2}
		(\overline\Omega\times[t_1,T])}
	\le
	\frac2\chi e^{\chi v^*/2}
	[z]_{C^{\alpha,\alpha/2}
		(\overline\Omega\times[t_1,T])}.
	\]
	Together with \(\|v\|_{L^\infty}\le v^*\), this proves
	\eqref{eq:z-v-holder}.
\end{proof}

\subsection{H\"older--Sobolev interpolation}
To exploit the H\"older control in the $u^p$-energy estimate, we use the
following H\"older--Sobolev interpolation inequality, which extends \cite[Lemma~2.2]{ChenCMS2025}.
\begin{lemma}
	\label{lem:interpolation}
	Let $\Omega\subset\mathbb R^N$ be a bounded domain with smooth boundary.
	For any $\alpha\in(0,1)$, $q\in(1,\infty)$, and
	\begin{equation}
		\theta\in\left(1/2,1/(2-\alpha)\right),\label{eq:theta-range}
	\end{equation}
	there exists $C=C(\Omega,\alpha,q,\theta)>0$ such that
	\begin{equation}
		\label{eq:interpolation}
		\|\nabla\varphi\|_{L^{2q}(\Omega)}
		\le
		C
		\|\varphi\|_{C^\alpha(\overline\Omega)}^\theta
		\|\varphi\|_{W^{2,q}(\Omega)}^{1-\theta}
	\end{equation}
	for every
	$\varphi\in C^\alpha(\overline\Omega)\cap W^{2,q}(\Omega)$.
\end{lemma}
\begin{proof}
	Set
	\begin{equation}
		\label{eq:s-r-interp}
		\sigma:=2-\frac1\theta,
		\qquad
		r:=\frac{2\theta q}{2\theta-1}.
	\end{equation}
	The restriction \eqref{eq:theta-range} ensures that
	\[
	0<\sigma<\alpha<1,
	\qquad
	1<r<\infty.
	\]
	Moreover,
	\begin{equation}
		\label{eq:interp-relations}
		1=\theta\sigma+2(1-\theta),
		\qquad
		\frac1{2q}=\frac\theta r+\frac{1-\theta}{q}.
	\end{equation}
	The Gagliardo--Nirenberg theorem
	\cite[Theorem~1(A)]{BrezisMironescu2018}, applied with
	\[
	(s_1,p_1)=(\sigma,r),
	\qquad
	(s,p)=(1,2q),
	\qquad
	(s_2,p_2)=(2,q),
	\]
	and with the interpolation parameter $\theta$, therefore yields
	\begin{equation}
		\label{eq:BM-interpolation}
		\norm{\varphi}_{W^{1,2q}(\Om)}
		\le C
		\norm{\varphi}_{W^{\sigma,r}(\Om)}^\theta
		\norm{\varphi}_{W^{2,q}(\Om)}^{1-\theta}.
	\end{equation}
	Indeed, the exceptional case in that theorem cannot occur because the
	higher-order exponent here is $p_2=q>1$.
	
	On the other hand, since $\sigma<\alpha$ and
	$\Om$ is bounded, by the definition of the fractional Sobolev seminorm in \cite[Definition~2.1]{BrezisMironescu2018}
	\begin{align*}
		[\varphi]_{W^{\sigma,r}(\Om)}^r
		&=\int_\Om\int_\Om
		\frac{|\varphi(x)-\varphi(y)|^r}{|x-y|^{N+\sigma r}}\,dx\,dy\\
		&\le [\varphi]_{C^\alpha(\overline\Om)}^r
		\int_\Om\int_\Om
		|x-y|^{-N+(\alpha-\sigma)r}\,dx\,dy\\
		&\le C\norm{\varphi}_{C^\alpha(\overline\Om)}^r.
	\end{align*}
	Together with the elementary estimate
	$\norm{\varphi}_{L^r(\Om)}\le
	C\norm{\varphi}_{C^\alpha(\overline\Om)}$, the above estimate gives
	\begin{equation}
		\label{eq:Calpha-Wsr}
		\norm{\varphi}_{W^{\sigma,r}(\Om)}
		\le C\norm{\varphi}_{C^\alpha(\overline\Om)}.
	\end{equation}
	Finally, \eqref{eq:interpolation} follows from  \eqref{eq:BM-interpolation}, \eqref{eq:Calpha-Wsr}, and the fact that
	$\norm{\nabla\varphi}_{L^{2q}}\le
	\norm{\varphi}_{W^{1,2q}}$.
\end{proof}

\begin{corollary}
	\label{cor:badterm}
	Let \(p>1\), and suppose that
	\[
	\|v\|_{C^\alpha(\overline\Omega)}\le K.
	\]
	Then, for every \(\varepsilon>0\), there exists
	\(C_{\varepsilon,p,K}
	=
	C(\varepsilon,p,K,\Omega,\alpha,\theta)\) such that
	\begin{equation}
		\label{eq:badterm}
		\int_\Omega u^p|\nabla v|^2\,dx
		\le
		\varepsilon\int_\Omega u^{p+1}\,dx
		+\varepsilon\|v\|_{W^{2,p+1}(\Omega)}^{p+1}
		+C_{\varepsilon,p,K}.
	\end{equation}
\end{corollary}

\begin{proof} Fix \(\theta\in(1/2,1/(2-\alpha))\) and let \(q=p+1\).
By H\"older's inequality and \cref{lem:interpolation},
\begin{align*}
 \int_\Om u^p|\nabla v|^2
 \le\norm{u}_{L^q}^p\norm{\nabla v}_{L^{2q}}^2
 \le C(K)
 \left(\int_\Om u^q\right)^{p/q}
 \left(\norm{v}_{W^{2,q}}^q\right)^{2(1-\theta)/q}.
\end{align*}
Since $1-\theta\in(0,1/2)$,
\[
 \frac pq+\frac{2(1-\theta)}q
 =\frac{p+2(1-\theta)}{p+1}<1,
\]
an application of  Young's inequality then yields \eqref{eq:badterm}.
\end{proof}

\subsection{Completion of boundedness}

Combining the preceding mixed-term estimate with weighted maximal
regularity for the signal equation yields uniform large-$L^p$ bounds
for the density.

\begin{proposition}
	\label{prop:Lp}
	Assume \(0<\tau<1\) and \eqref{eq:threshold}. For every
	\(p>N\) and every \(t_2\in(t_0,T_{\max})\), there exists
	\(C_p>0\) such that
	\begin{equation}
		\label{eq:Lp-uniform}
		\sup_{t\in[t_2,T_{\max})}
		\|u(t)\|_{L^p(\Om)}
		\le C_p.
	\end{equation}
\end{proposition}

\begin{proof}
	Fix \(p>N\) and set \(q:=p+1\). By
	\cref{prop:z-v-holder}, with \(t_1=t_2\), there exists
	\(K_\alpha>0\) such that
	\begin{equation}
		\label{eq:spatial-holder-uniform}
		\sup_{t\in[t_2,T_{\max})}
		\|v(t)\|_{C^\alpha(\overline\Om)}
		\le K_\alpha .
	\end{equation}
	
	Testing the first equation in \eqref{eq:system} by \(u^{p-1}\)
	and using Young's inequality, we obtain
	\begin{align}
		\frac1p\frac{d}{dt}\int_\Om u^p
		&+\frac{p-1}{2}\int_\Om u^{p-2}|\nabla u|^2
		+\mu\int_\Om u^{p+1} \notag\\
		&\le
		\frac{\chi^2(p-1)}2
		\int_\Om u^p|\nabla v|^2
		+\lambda_+\int_\Om u^p .
		\label{eq:p-energy}
	\end{align}
	By the Neumann elliptic estimate, \eqref{eq:spatial-holder-uniform},
	and \cref{cor:badterm}, for every \(\varepsilon>0\),
	\[
	\int_\Om u^p|\nabla v|^2
	\le
	\varepsilon\int_\Om u^q
	+\varepsilon\|\Delta v\|_{L^q(\Om)}^q
	+C_{\varepsilon,p}.
	\]
	Together with
	\[
	\int_\Om u^p
	\le \eta\int_\Om u^q+C_{\eta,p},
	\]
	this yields constants \(c_0>0\), \(c_p>0\), and
	\(\varepsilon_0>0\) such that
	\begin{equation}
		\label{eq:yprime-U-D}
		y'(t)+c_0U(t)
		\le c_p\varepsilon D(t)+C_{\varepsilon,p}
	\end{equation}
	for \(0<\varepsilon\le\varepsilon_0\), where
	\[
	y(t):=\int_\Om u^p,\qquad
	U(t):=\int_\Om u^q,\qquad
	D(t):=\|\Delta v(t)\|_{L^q(\Om)}^q.
	\]
	
	Set \(\sigma:=q/\tau\). Since by Young's inequality
	\[
	\sigma y(t)\le \frac{c_0}{2}U(t)+C,
	\]
	we infer from \eqref{eq:yprime-U-D} that
	\[
	y'(t)+\sigma y(t)+\frac{c_0}{2}U(t)
	\le c_p\varepsilon D(t)+C_{\varepsilon,p}.
	\]
	Hence, for \(t\in[t_2,T_{\max})\),
	\begin{align}
		y(t)+\frac{c_0}{2}
		\int_{t_2}^t e^{-\sigma(t-s)}U(s)\,ds
		&\le
		e^{-\sigma(t-t_2)}y(t_2) \notag\\
		&\quad
		+c_p\varepsilon
		\int_{t_2}^t e^{-\sigma(t-s)}D(s)\,ds
		+C_{\varepsilon,p}.
		\label{eq:y-weighted}
	\end{align}
	
	Since \(t_2>0\), we have
	\(v(t_2)\in W^{2,q}(\Om)\) and
	\(\partial_\nu v(t_2)=0\). Thus,
	\cite[Lemma~2.2]{YangCaoJiangZheng2015} gives
	\begin{align}
		\int_{t_2}^t e^{-\sigma(t-s)}D(s)\,ds
		&\le
		K_q\int_{t_2}^t e^{-\sigma(t-s)}U(s)\,ds \notag\\
		&\quad
		+C e^{-\sigma(t-t_2)}
		\left(
		\|v(t_2)\|_{L^q}^q+
		\|\Delta v(t_2)\|_{L^q}^q
		\right).
		\label{eq:D-by-U}
	\end{align}
	Now \(p,q,\tau\), and hence \(K_q\), are fixed. Choosing
	\(\varepsilon>0\) so small that
	\[
	c_p\varepsilon K_q<\frac{c_0}{4},
	\]
	and substituting \eqref{eq:D-by-U} into
	\eqref{eq:y-weighted}, we obtain
	\[
	y(t)\le C_p
	\qquad
	\text{for all }t\in[t_2,T_{\max}).
	\]
	This completes the proof.
\end{proof}

Choosing $p>N+2$ in the preceding estimate and applying standard
parabolic smoothing completes the case $0<\tau<1$.
\begin{proposition}
	\label{prop:completion-tau<1}
	Under the assumptions of \cref{thm:main}, if \(0<\tau<1\), then
	\(T_{\max}=\infty\) and \eqref{eq:mainbounded} holds.
\end{proposition}

\begin{proof}
	Choose \(p>N+2\) and \(t_2\in(t_0,T_{\max})\). By
	\cref{prop:Lp},
	\begin{equation}
		\label{eq:Lp-final}
		\sup_{t\in[t_2,T_{\max})}
		\|u(t)\|_{L^p(\Om)}<\infty.
	\end{equation}
	Standard parabolic regularity theory  then implies that, for every
	\(t_3\in(t_2,T_{\max})\),
	\begin{equation}
		\label{eq:gradv-infty}
		\sup_{t\in[t_3,T_{\max})}
		\|\nabla v(t)\|_{L^\infty(\Om)}<\infty.
	\end{equation}

Set
\[
F:=-\chi u\nabla v,
\qquad
G:=\lambda u-\mu u^2.
\]
Then
\[
F\in L^\infty((t_3,T_{\max});L^p(\Omega;\mathbb R^N)),
\qquad
G\in L^\infty((t_3,T_{\max});L^{p/2}(\Omega)),
\]
and \(F\cdot\nu=0\) on \(\partial\Omega\). Since
\[
p>N+2,
\qquad
\frac p2>\frac{N+2}{2},
\]
all assumptions of \cite[Lemma~A.1]{TaoWinkler2012} are fulfilled
(with \(D\equiv1\)).	
	
	Now we may apply \cite[Lemma~A.1]{TaoWinkler2012} to the first equation
	in \eqref{eq:system} to obtain that
	\[
	\sup_{t\in[t_3,T_{\max})}
	\|u(t)\|_{L^\infty(\Om)}<\infty.
	\]
	
Thus by \cref{lem:local}, we have \(T_{\max}=\infty\).
	Combining the above estimates with local boundedness on
	\([0,t_3]\) proves \eqref{eq:mainbounded}.
\end{proof}

\medskip

\begin{proof}[Proof of \cref{thm:main}]
	If $\tau\ge1$, global existence and \eqref{eq:mainbounded} follow from
	\cref{prop:tau-geq1}. The same proposition yields
	\eqref{eq:explicitabsorbing} when $\tau>1$ and
	\eqref{eq:tau1absorbing} when $\tau=1$. If $0<\tau<1$, global existence and
	\eqref{eq:mainbounded} follow from \cref{prop:completion-tau<1}.
\end{proof}

\section*{Acknowledgments}
This work was supported by the National Natural Science Foundation of China
(Grant No.~12271505).


\begin{thebibliography}{99}

\bibitem{BellomoBellouquidTaoWinkler2015}
N.~Bellomo, A.~Bellouquid, Y.~Tao, and M.~Winkler,
\newblock Toward a mathematical theory of Keller--Segel models of pattern
formation in biological tissues,
\newblock \emph{Math. Models Methods Appl. Sci.} \textbf{25} (2015),
1663--1763.



\bibitem{BrezisMironescu2018}
H.~Brezis and P.~Mironescu,
\newblock Gagliardo--Nirenberg inequalities and non-inequalities: the full
story,
\newblock \emph{Ann. Inst. H. Poincar\'e C Anal. Non Lin\'eaire}
\textbf{35} (2018), no.~5, 1355--1376.
\newblock \href{https://doi.org/10.1016/j.anihpc.2017.11.007}
{doi:10.1016/j.anihpc.2017.11.007}.




\bibitem{ChenCMS2025}
L.~Chen,
\newblock \emph{Global well-posedness for a two-dimensional Navier--Stokes--Cahn--Hilliard--Boussinesq system with singular potential},
\newblock \emph{Comm. Math. Sci.} \textbf{23} (2025), 509--540.




\bibitem{Lieberman2005}
G.~M. Lieberman,
\newblock \emph{Second Order Parabolic Differential Equations},
\newblock revised ed., World Scientific, Hackensack, NJ, 2005.

\bibitem{LinMu2016}
K.~Lin and C.~Mu,
\newblock Global dynamics in a fully parabolic chemotaxis system with
logistic source,
\newblock \emph{Discrete Contin. Dyn. Syst.}
\textbf{36} (2016), 5025--5046.
\newblock \href{https://doi.org/10.3934/dcds.2016018}
{doi:10.3934/dcds.2016018}.

\bibitem{TaoWinkler2012}
Y.~Tao and M.~Winkler,
\newblock Boundedness in a quasilinear parabolic--parabolic
Keller--Segel system with subcritical sensitivity,
\newblock \emph{J. Differential Equations} \textbf{252} (2012),
692--715.

\bibitem{TaoWinkler2015ZAMP}
Y.~Tao and M.~Winkler,
\newblock Boundedness and decay enforced by quadratic degradation in a
three-dimensional chemotaxis--fluid system,
\newblock \emph{Z. Angew. Math. Phys.} \textbf{66} (2015), 2555--2573.

\bibitem{TelloWinkler2007}
J.~I. Tello and M.~Winkler,
\newblock A chemotaxis system with logistic source,
\newblock \emph{Comm. Partial Differential Equations} \textbf{32} (2007),
849--877.

\bibitem{Winkler2010}
M.~Winkler,
\newblock Boundedness in the higher-dimensional parabolic--parabolic
chemotaxis system with logistic source,
\newblock \emph{Comm. Partial Differential Equations} \textbf{35} (2010),
1516--1537.



\bibitem{Xiang2018JMAA}
T.~Xiang,
\newblock How strong a logistic damping can prevent blow-up for the minimal
Keller--Segel chemotaxis system?,
\newblock \emph{J. Math. Anal. Appl.} \textbf{459} (2018), 1172--1200.
\newblock \href{https://doi.org/10.1016/j.jmaa.2017.11.022}
{doi:10.1016/j.jmaa.2017.11.022}.

\bibitem{Xiang2018SIAM}
T.~Xiang,
\newblock Chemotactic aggregation versus logistic damping on boundedness in
the 3D minimal Keller--Segel model,
\newblock \emph{SIAM J. Appl. Math.} \textbf{78} (2018), 2420--2438.
\newblock \href{https://doi.org/10.1137/17M1150475}
{doi:10.1137/17M1150475}.


\bibitem{YangCaoJiangZheng2015}
C.~Yang, X.~Cao, Z.~Jiang and S.~Zheng,
\newblock Boundedness in a quasilinear fully parabolic Keller--Segel
system of higher dimension with logistic source,
\newblock \emph{J. Math. Anal. Appl.} \textbf{430} (2015), 585--591.

\bibitem{Zheng2017}
J.~Zheng,
\emph{Boundedness and global asymptotic stability of constant equilibria
	in a fully parabolic chemotaxis system with nonlinear logistic source},
J. Math. Anal. Appl. \textbf{450} (2017), 1047--1061.

\bibitem{ZhengEtAl2018}
J.~Zheng, Y.~Li, G.~Bao and X.~Zou,
\newblock A new result for global existence and boundedness of solutions to a
parabolic--parabolic Keller--Segel system with logistic source,
\newblock \emph{J. Math. Anal. Appl.} \textbf{462} (2018), 1--25.

\end{thebibliography}
\end{document}